\documentclass[12pt,a4paper,enabledeprecatedfontcommands]{scrartcl}
\usepackage[utf8]{inputenc}
\usepackage[T1]{fontenc}
\usepackage{lmodern}

\usepackage[british]{babel}

\usepackage{doi,url}
\usepackage[numbers]{natbib}
\usepackage{etoolbox,letltxmacro,xifthen,ifdraft} % reinforce the environment
\usepackage[dvipsnames]{xcolor}
\usepackage{amsmath,amssymb,enumitem,bbm,mathrsfs,amsthm,aliascnt,mathtools}
\usepackage[retainorgcmds]{IEEEtrantools}
\usepackage{hyperref} % should be last
\hypersetup{final,hidelinks}
\usepackage[atend]{bookmark} % should be laster

\providecommand{\email}[1]{\href{mailto:#1}{\nolinkurl{#1}}}

\setlist[enumerate,1]{label={(\roman*)}}
\setlist[enumerate,2]{label={(\alph*)}}
\setlist[enumerate,3]{label={(\Roman*)}}

\newcommand{\newsstheorem}[2]{
  \newaliascnt{#1}{dummy}
  \newtheorem{#1}[#1]{#2}
  \aliascntresetthe{#1}
  \expandafter\def\csname #1autorefname\endcsname{#2}
}
\numberwithin{dummy}{section}
\theoremstyle{plain}
  \newsstheorem{theorem}{Theorem}
  \newsstheorem{proposition}{Proposition}
  \newsstheorem{corollary}{Corollary}
  \newsstheorem{lemma}{Lemma}
\theoremstyle{definition}
  \newsstheorem{definition}{Definition}
  \newsstheorem{example}{Example}
  \newsstheorem{assumption}{Assumption}
\theoremstyle{remark}
  \newsstheorem{remark}{Remark}

\newenvironment{eqnarr*}{\begin{IEEEeqnarray*}{rCl}}{\end{IEEEeqnarray*}\ignorespacesafterend}

\newcommand\RR{\mathbb{R}}

\newcommand\PP{\mathbb{P}}
\newcommand\QQ{\mathbb{Q}}
\newcommand\EE{\mathbb{E}}

\newcommand\Indic[1]{\Ind_{\{#1\}}}

\newcommand\e{{\rm e}}

\newcommand\Ind{\mathbbm{1}}

\newcommand\from{\colon}

\makeatletter \newcommand\mathof[1]{{\operator@font#1}} \makeatother

\newcommand\dd{\mathof{d}}

\DeclarePairedDelimiter\abs{\lvert}{\rvert}

\DeclarePairedDelimiterX\ip[2]{\langle}{\rangle}{#1,#2}
\begin{document}

\title{A probabilistic approach\\ to spectral analysis of\\ growth-fragmentation equations}
\author{Jean Bertoin\footnote{Insitute of Mathematics, University of Zurich, Switzerland}  \and Alexander R. Watson\footnote{School of Mathematics, University of Manchester, UK}}
\date{\unskip}
\maketitle 
\thispagestyle{empty}

\begin{abstract}
  \noindent
  The growth-fragmentation equation describes a system
  of growing and dividing particles,
  and arises in models of cell division, protein
  polymerisation and even telecommunications protocols.
  Several important questions about the equation concern the
  asymptotic behaviour of solutions at large times: at what rate do they
  converge to zero or infinity, and what does the asymptotic
  profile of the solutions look like? Does the rescaled
  solution converge to its asymptotic profile at
  an exponential speed? These questions have traditionally
  been studied using analytic techniques   such as entropy
  methods or splitting of operators. In this work, we present
  a probabilistic approach
  to the study of this asymptotic behaviour.
  We use a Feynman--Kac formula to relate the solution
  of the growth-fragmentation equation to the semigroup
  of a Markov process, and characterise the rate of decay or growth
  in terms of this process. We then identify the spectral radius and the
  asymptotic profile in terms of a related Markov process,
  and give a spectral interpretation in
  terms of the growth-fragmentation operator and its dual.
  In special cases, we
  obtain exponential convergence.
\end{abstract} 
{\small 
}

\section{Introduction}
\label{s:intro}

This work studies the asymptotic behaviour of
solutions to the growth-fragmentation equation using
probabilistic methods.
The growth-fragmentation arises from mathematical models of 
biological phenomena such as cell division \cite[\S 4]{Per07} and 
protein polymerization \cite{GPW06},
as well as in telecommunications \cite{CMP10}.
The equation
describes the evolution of the density $u_t(x)$ of
particles of mass $x>0$  at time $t\geq 0$, in a system whose
dynamics are given as follows.
Each particle grows at a certain rate depending
on its mass and experiences `dislocation events',
again at a rate depending on its mass. 
At each such event, it
splits into smaller particles in such a way that the total mass
is conserved.
The growth-fragmentation equation is a partial integro-differential
equation and can be expressed in the form
\begin{equation}\label{e:gfe}
  \partial_t u_t(x) +  \partial_x(c(x)u_t(x))
  = \int_x^{\infty} u_t(y) k(y,x) \dd y - K(x)u_t(x),
\end{equation}
where $c \from (0,\infty)\to (0,\infty)$ is a continuous positive 
function specifying the growth rate, 
$k\from (0,\infty)\times (0,\infty)\to \RR_+$ is a so-called fragmentation 
kernel, and the initial condition $u_0$ is prescribed.
In words,  $k(y,x)$ represents  the rate at which a particle 
with size $x$ appears as the result of the dislocation of a particle
with mass $y>x$.
More precisely,  the fragmentation kernel
 fulfills 
$$k(x,y)=0 \text{ for $y>x$, and } \int_0^x yk(x,y)\dd y=xK(x).$$
The first requirement stipulates that after the
dislocation of a particle, only particles with smaller masses
can arise. The second reflects the conservation of mass
at dislocation events,
and gives the interpretation of
$K(x)$ as the total rate of dislocation of particles
with size $x$.

% \medskip\noindent
This equation has been studied extensively over many years.
A good introduction to growth-fragmentation equations
and related equations in biology can be found in
the monographs of \citet{Per07}
and \citet{EN00}, and a major issue concerns the
asymptotic behaviour of solutions $u_t$. Typically, one wishes to find
a constant $\rho\in\RR$, the \emph{spectral radius},
for which $\e^{-\rho t}u_t$ converges, in some suitable space,
to a so-called \emph{asymptotic profile} $v$. Ideally, we would also
like to have some information about the \emph{rate of convergence};
that is, we would like to find some $r>0$ with the property
that $\e^{-r t}(\e^{-\rho t}u_t - v)$ converges to zero.

For such questions, a key step in finding $\rho$ is the
spectral analysis of the growth-fragmentation operator 
\begin{equation}\label{e:A}
  \mathcal{A}f(x) = c(x)f'(x) + \int_0^x f(y) k(x,y) \dd y - K(x) f(x), \qquad x>0,  
\end{equation}
which is defined for smooth compactly supported $f$, say.

Indeed, observe first that the weak form of
the growth-fragmentation equation \eqref{e:gfe}
is given by 
\begin{equation}\label{e:mu}
  \frac{\mathrm {d}}{\mathrm{d} t} \ip{u_t}{f} = \ip{u_t}{\mathcal{A}f},
\end{equation}
where we use the notation $\ip{\mu}{g} \coloneqq \int g(x) \, \mu(\dd x)$
for any measure $\mu$ and function $g$ on the same space, and $\ip{f}{g}\coloneqq \ip{\mu}{g}$ with $\mu(\dd x) = f(x)\dd x$
when $f\geq 0$ is a measurable function. 
Under some simple assumptions that we will specify shortly,
there exists a unique semigroup $(T_t)_{t\geq 0}$,
defined on a certain Banach space of functions on $(0,\infty)$,
whose infinitesimal 
generator extends $\mathcal{A}$. Then, the solutions $u_t$
of \eqref{e:mu} have the representation
$$\ip{u_t}{f}=\ip{u_0}{T_tf}.$$
Several authors  have shown  the existence of a positive eigenfunction 
associated to the first eigenvalue of the dual operator $\mathcal{A}^*$ 
and established exponential convergence of the solution to an asymptotic profile,
under certain assumptions on $c$ and $K$.
Since the literature is considerable, we refer only to a few works which are quite close
to our assumptions or approach: \citet{CCM11} study the case of linear growth
and $K$ bounded by a power function,
via entropy methods; \citet{MS16} use a splitting technique in order to derive
a Krein--Rutman theorem, which is effective when $c$ is constant and $K$
is zero in some neighbourhood of $0$; and \citet{BPR12} study a situation
in which particle sizes are bounded, and do so via an interesting connection with 
stochastic semigroups. 
Moreover, \citet{CDG} investigate the dependence of the leading eigenvalue (i.e. 
 the spectral radius), 
and the corresponding eigenvector on the coefficients of the equation; 
and \citet{Bou} studies a conservative
version of the equation using a Markov process approach similar to ours.

% \medskip\noindent
The purpose of this work is to show the usefulness of stochastic
methods in this setting.
We have not attempted to find the most general conditions, but rather
to demonstrate the benefits of the probabilistic approach.
For the sake of simplicity and conciseness, we shall restrict our attention 
to the case when the growth rate is bounded from above by a linear function, namely
\begin{equation}\label{e:c-bound}
  \| \underline c\|_{\infty}\coloneqq \sup_{x>0} c(x)/x <\infty,
\end{equation}
and we shall shortly make some further technical assumptions on the 
fragmentation kernel $k$. We stress that the techniques developed in this work can be adapted to
deal with other types of growth and fragmentation rates of interest which have been considered in preceding works. 

In short, we will obtain probabilistic representations of the main quantities
of interest (the semigroup $T_t$, the spectral radius $\rho$, the asymptotic profile $v$, and so on)
in terms of a certain Markov process  with values in $(0,\infty)$. 
Specifically, even though $(T_t)_{t\geq 0}$ is not a Markovian 
(i.e., contraction) semigroup, the operator
$${\mathcal G}f(x) \coloneqq c(x)f'(x)+ \int_0^x(f(y)-f(x)) \frac{y}{x} k(x,y) \, \dd y$$
\emph{is} the infinitesimal generator of a Markovian semigroup, and this
operator is closely connected to $\mathcal{A}$. 

To be precise, comparing ${\mathcal A}$ and ${\mathcal G}$ allows us 
to express the semigroup $T_t$ via a so-called Feynman--Kac formula:
\begin{equation}\label{e:fk}
  T_tf(x) = x \EE_x\left( \mathcal{E}_t \frac{ f(X_t)}{X_t}\right),
  \qquad t\ge 0, \quad x>0,
\end{equation}
where
$X$ is the Markov process with infinitesimal generator $\mathcal{G}$,
$\PP_x$ and $\EE_x$ represent respectively the probability measure and
expectation under which $X$ starts at $X_0 = x$, and
$$
  \mathcal{E}_t\coloneqq \exp\left(\int_0^t \frac{c(X_s)}{X_s}  \dd s \right),
  \qquad t\ge 0.
$$
Even though the formula \eqref{e:fk} is not very explicit in general, 
we can use it to say quite a lot about the behaviour of  $T_t$ as $t\to \infty$.

In this direction, a fundamental role is played by the function $L_{x,y}: 
\RR \to (0,\infty]$ defined as the 
Laplace transform 
\begin{equation}\label{e:lbar}
 L_{x,y}(q) \coloneqq \EE_{x}\left( \e^{-q H(y)} \mathcal{E}_{H(y)}, H(y)<\infty\right),
\end{equation}
where $H(y)$ denotes the first hitting time of  $y$, 
Indeed, we identify  a first quantity of importance in the study of the large time
behaviour of $(T_t)_{t\geq 0}$, namely the \emph{spectral radius},
as
\begin{equation}\label{e:lambda}
  \rho   \coloneqq   \inf   \left\{   q\in\RR:   L_{x,x}(q) < 1   \right \},
\end{equation}
where  $x>0$ is arbitrary.
The quantity $\rho$ is sometimes called the `Malthus exponent' in
the literature on growth-fragmentation.

Next, we shall focus on the case where 
\begin{equation} \label{e:Hlambda}
  L_{x,x}(\rho) = 1
\end{equation}
for some (and then all) $x>0$, and, for arbitrary fixed $x_0>0$, set 
$$\ell(x) = L_{x,x_0}(\rho).$$
Then, 
the function
$$x\mapsto \bar \ell(x) \coloneqq  x\ell(x)$$
can be viewed as an eigenfunction of  ${\mathcal A}$ with eigenvalue
$\rho$, whenever the function $\ell$ is bounded. 
Furthermore, provided that the function $q \mapsto L_{x,x}(q)$ possesses a finite
right-derivative at $\rho$ for some (and then all) $x>0$,  
the absolutely continuous measure
\begin{equation} \label{e:statm}
\nu(\dd x)\coloneqq  \frac{\dd x}{\bar \ell(x) c(x) |L'_{x,x}(\rho)|} , \qquad x>0, 
 \end{equation}
is an eigenmeasure of the dual operator ${\mathcal A}^*$, with eigenvalue $\rho$
(at least under some further technical conditions).

Finally, one can describe the asymptotic behaviour of the fragmentation semigroup 
as follows. For every $x>0$ and continuous function $f:(0,\infty)\to \RR$
with compact support, one has
\begin{equation}\label{e:nu}
\lim_{t\to \infty} \e^{-\rho t} T_tf(x)= \bar \ell(x)  \ip{\nu}{f}.\end{equation}
In certain concrete situations, we can furthermore demonstrate
exponential convergence, using classical probabilistic techniques.

Technically, the cornerstone of our analysis is that the assumption \eqref{e:Hlambda}
 enables us to define a remarkable martingale multiplicative functional ${\mathcal M}$ 
 of $X$. In turn, by classical  change of probabilities, ${\mathcal M}$ yields 
 another Markov process $(Y_t)_{t\geq 0}$ that is always {\it recurrent}. 
 Using ergodic theory for recurrent  Markov processes then readily leads
 to the large time asymptotic behaviour of the growth-fragmentation semigroup
 mentioned above.

% \medskip\noindent
The formulas above may seem somewhat cryptic, 
but could nonetheless be  useful in applications, for instance
as the basis of a Monte Carlo method for computing the spectral
radius and its corresponding eigenfunction and dual eigenmeasure.
There are well-established algorithms for efficiently 
simulating Markov processes,
and the process $X$ which appears here falls within the even
nicer class of `piecewise deterministic' Markov processes.
This simulation is probably less costly than numerical estimation of
the leading eigenvalue and corresponding eigenfunctions of
${\mathcal A}$ and its dual, at least when the spectral gap
is small or absent.

% \medskip\noindent
The remainder of this article is organised as follows.
In \autoref{s:fk}, we make precise the relationship
between the operators $\mathcal{A}$ and $\mathcal{G}$,
and derive the Feynman--Kac formula \eqref{e:fk}. Along
the way, this establishes the existence and uniqueness
of solutions to \eqref{e:mu}. In \autoref{s:malthus},
we identify the spectral radius
$\rho$ and give some simple bounds for this quantity.
Under the assumption \eqref{e:Hlambda}, we
give in \autoref{s:mult} a martingale $\mathcal{M}$ for
the process $X$, and apply it in order to show that
the function $\bar\ell$ is an eigenfunction of $\mathcal{A}$
with eigenvalue $\rho$.
We then use the martingale $\mathcal{M}$,
in \autoref{s:ergodic},
to transform $X$  into
another Markov process $Y$, by a classical change of measure.
The key point is that the process $Y$ is always
recurrent, and this leads to our main result,
\autoref{Th2}, which comes from the ergodic theory
of positive recurrent Markov processes. In this section,
we also show that $\nu$ is an eigenmeasure of $\mathcal{A}^*$.
Finally, in \autoref{s:levy}, we specialise our results to the
case where the growth rate is linear, that is $c(x)=ax$,
and give more explicit results, including
criteria for exponential convergence to the asymptotic profile.
We also study in some detail a special case where
the strongest form of convergence does not hold.

\section{Feynman-Kac representation of the semigroup}
\label{s:fk}

Our main task in this section is to derive a representation
of the semigroup $T_t$ solving the growth-fragmentation
equation, using a Feynman--Kac formula. We begin
by introducing some notation and listing the assumptions which
will be required for our results.

We write ${\mathcal C}_b$ for the Banach space  of 
continuous and bounded functions
\linebreak % line-edit for arXiv version
$f: (0,\infty)\to \RR$, 
endowed with the supremum norm $\|\cdot \|_{\infty}$. 
It will be further convenient to set 
\linebreak % line-edit for for arXiv version
$\bar f(x)= xf(x)$ for every $f\in{\mathcal C}_b$ and $x>0$, 
and define $\bar {\mathcal C}_b=\{\bar f: f\in {\mathcal C}_b\}$. 
Analogously, we set $\underline f(x)= x^{-1}f(x)$.

Recall our assumption \eqref{e:c-bound} that the growth rate
$c$ is continuous and is bounded from above by a linear function,
that is, in our notation, $\underline c\in {\mathcal C}_b$. 
We further set 
$$\bar k(x,y)\coloneqq  \frac{y}{x} k(x,y),$$
and assume that 
\begin{equation}\label{H8}
  x\mapsto \bar k(x,\cdot)
  \text{ is a continuous bounded map from $(0,\infty)$ to $L^1(\dd y)$.}
\end{equation} 

Recall furthermore that the operator ${\mathcal A}$ is
defined by \eqref{e:A}; in fact, it will be more
convenient for us to consider 
$$\bar {\mathcal{A}}f (x)= \frac{1}{x} \mathcal{A}\bar f(x),$$
which can be written as
\begin{equation}\label{e:bara}
  \bar {\mathcal{A}}f(x)
  =
  c(x) f'(x) + \int_0^x ( f(y) - f(x) )\, \bar k(x,y) \dd y + \underline c(x)  f(x).
\end{equation}
We view  $\bar {\mathcal A}$ as an operator on ${\mathcal C}_b$
whose domain ${\mathcal D}(\bar {\mathcal A})$ contains the space 
of bounded continuously differentiable  functions $f$  such that 
$c f'$ bounded. Equivalently, $ {\mathcal A}$ is seen as an 
operator on $\bar {\mathcal C}_b$ with domain
${\mathcal D}( {\mathcal A})=\{\bar f: f\in {\mathcal D}(\bar {\mathcal A})\}$. 
The following lemma, ensuring the existence and uniqueness 
of semigroups $\bar{T}_t$ and $T_t$
with infinitesimal generators 
$\bar {\mathcal A}$ and $ {\mathcal A}$ respectively,
relies on standard arguments. 

\begin{lemma}\label{L0}
  Under the assumptions above, we have:
  \begin{enumerate}
    \item There exists a unique positive strongly continuous semigroup 
    $(\bar T_t)_{t\geq 0}$ on ${\mathcal C}_b$ whose infinitesimal 
    generator coincides with
    $\bar {\mathcal A}$ on the space of bounded continuously 
    differentiable  functions $f$  with  $c f'$ bounded.
 
    \item As a consequence, the identity
    $$ T_t \bar f(x) = x \bar T_t f(x), \quad f\in  {\mathcal C}_b \text{ and } x>0$$
    defines the unique positive strongly continuous semigroup $(T_t)_{t\geq 0}$
    on $\bar {\mathcal C}_b$ with infinitesimal generator 
    ${\mathcal A}$. 
  \end{enumerate}
\end{lemma}
\begin{proof} 
  Recall that $\underline c\in{\mathcal C}_b$ and consider
  first the operator
  $\tilde {\mathcal{A}}f\coloneqq  \bar {\mathcal{A}}f - \|\underline c\|_{\infty}f$,
  that is,
  $$
    \tilde {\mathcal{A}}f(x) 
    =
    c(x) f'(x) + \int_0^x ( f(y) - f(x) )\, \bar k(x,y) \dd y 
    - (\|\underline c\|_{\infty} -\underline c(x))  f(x),
  $$ 
  which is defined for $f$ bounded and continuously differentiable with $c f'$ bounded. 
  Plainly $\|\underline c\|_{\infty} -\underline c\geq 0$, and
  we may view $\tilde {\mathcal{A}}$ as the infinitesimal generator 
  of a  (sub-stochastic, i.e., killed) Markov process $\tilde X$ on $(0,\infty)$. 
  More precisely, it follows from our assumptions 
  (in particular, recall that by \eqref{H8}, the jump kernel $\bar k$ is bounded)
  that the martingale problem for 
  $\tilde {\mathcal{A}}$ is well-posed;
  this can be shown quite simply using \cite[Theorem 8.3.3]{EK-mp}, for instance.
  The transition probabilities of $\tilde X$ yield a positive 
  contraction semigroup  on ${\mathcal C}_b$, say $(\tilde T_t)_{t\geq 0}$,
  that   has  infinitesimal generator $\tilde {\mathcal{A}}$. 
  Then $\bar T_tf\coloneqq \exp(t\|\underline c\|_{\infty}) \tilde T_t f$ 
  defines a positive strongly continuous semigroup on ${\mathcal C}_b$ with
  infinitesimal generator $\bar {\mathcal{A}}$.

  Conversely, if  $(\bar T_t)_{t\geq 0}$ is a positive strongly continuous 
  semigroup on ${\mathcal C}_b$ with infinitesimal generator 
  $\bar {\mathcal A}$, then 
  $$
    \frac{\dd }{\dd t} \bar T_t{\bf 1} 
    = 
    \bar T_t \bar {\mathcal A}{\bf 1} \leq \|\underline c\|_{\infty} \bar T_t{\bf 1},
  $$
  where $\mathbf{1}$ is the constant function with value $1$.
  It follows that 
  $\|\bar T_t f\|_{\infty} \leq \exp(t\|\underline c\|_{\infty})\| f \|_{\infty}$ 
  for all $t\geq 0$ and $f\in{\mathcal C}_b$, and
  $\tilde T_t\coloneqq \exp(t\|\underline c\|_{\infty})\bar T_t$
  defines a positive strongly continuous semigroup on ${\mathcal C}_b$ with
  infinitesimal generator $\tilde {\mathcal{A}}$. 
  The well-posedness of the martingale problem for $\tilde {\mathcal{A}}$ 
  ensures the uniqueness of $(\tilde T_t)_{t\geq 0}$, and thus of $(\bar T_t)_{t\geq 0}$.

  The second assertion follows from a well-known and easy to check formula 
  for multiplicative transformation of semigroups. 
\end{proof}

Although neither $(T_t)_{t\geq 0}$ or $(\bar T_t)_{t\geq 0}$ 
is a contraction semigroup, they both bear a simple relation 
to a certain Markov process with state space $(0,\infty)$, which we now introduce. 
The operator
\begin{equation}\label{e:G}
  {\mathcal G}f(x)
  \coloneqq 
  \bar {\mathcal{A}}f(x) 
  - \underline c(x)f(x)
  =
  c(x)f'(x)+\int_0^x(f(y)-f(x)) \bar k(x,y) \, \dd y,
\end{equation}
with domain $\mathcal{D}(\mathcal{G}) = {\mathcal D}(\bar {\mathcal A})$
is indeed the infinitesimal generator of a conservative (unkilled)
Markov process
$X=(X_t)_{t\geq 0}$, 
and in fact, it is easy to check, again using
\cite[Theorem 8.3.3]{EK-mp}, that the martingale problem
$$
  f(X_t)-\int_0^t {\mathcal G}f(X_s) \dd s
  \quad \text{ is a martingale for every ${\mathcal C}^{1}$ function $f$ with compact support}
$$
is well-posed.
In particular,  the law of $X$ is characterized by ${\mathcal G}$.
We write $\PP_x$ for the law of $X$ started from $x>0$, 
and $\EE_x$ for the corresponding mathematical expectation.

The process $X$ belongs to the class of \emph{piecewise deterministic Markov processes}
introduced by \citet{Dav-pdmp}, meaning that any path $t\mapsto X_t$
follows the
deterministic flow $\dd x(t)= c(x(t))\dd t$, up to a random time at which
it makes its first (random) jump. 
Note further that,  since
$$
  \int_0^{1} \frac{\dd x}{c(x)} = \int_1^{\infty} \frac{\dd x}{c(x)} =  \infty,
$$
$X$  can neither enter from $0$ nor reach $0$ or $\infty$ in finite time.
Finally, it is readily checked that $X$ has the Feller property, in the
sense that its transition probabilities depend continuously on the starting point. 
For the sake of simplicity, we also assume that $X$ is irreducible; this
means that, for every starting point $x>0$, the probability that the Markov process
started from $x$ hits a given target point $y>0$ is strictly positive.
Because $X$ is piecewise deterministic and has only downwards jumps, 
this can be ensured by a simple non-degeneracy assumption on the
fragmentation kernel $k$.

Lemma \ref{L0}(ii) and equation \eqref{e:G} prompt us to consider the exponential functional
$$
  {\mathcal E}_t
  \coloneqq  
  \exp\left(\int_0^t \underline  c(X_s)  \dd s \right),\qquad t\geq 0.
$$
We note the uniform bound ${\mathcal E}_t\leq \exp(t\|\underline c\|_{\infty})$,
and also observe, from the decomposition of the trajectory of $X$ at 
its jump times, that there is the identity
$$
  {\mathcal E}_t=\frac{X_t}{X_0} \prod_{0<s\leq t} \frac{X_{s-}}{X_s}.
$$

The point in introducing the elementary transformation and notation above is 
that it yields a Feynman-Kac representation of the growth-fragmentation semigroup,
which appeared as equation \eqref{e:fk} in the introduction:

\begin{lemma}\label{Le1}  
The growth-fragmentation semigroup $(T_t)_{t\geq 0}$ can be expressed in the  form 
$$
  T_tf(x) = x \EE_x\left( \mathcal{E}_t \underline f(X_t)\right) 
  = x \EE_x\left( \mathcal{E}_t \frac{ f(X_t)}{X_t}\right),
  \qquad f\in \bar{\mathcal C}_b.
$$
\end{lemma}

\begin{proof}
  Recall  from Dynkin's formula that for every $f\in{\mathcal D}(\bar {\mathcal A})$, 
  $$f(X_t)-\int_0^t {\mathcal G}f(X_s) \dd s\,, \qquad t\geq 0$$
  is a $\PP_x$-martingale for every $x>0$. 
  Since  $(\mathcal{E}_t)_{t\geq 0}$ is a process of bounded variation with 
  $\dd {\mathcal E}_t= \underline c(X_t) {\mathcal E}_t \dd t$,
   the integration by parts formula of stochastic calculus 
   \cite[Corollary~2 to Theorem~II.22]{Pro-sc}
   shows that
  $$
    {\mathcal E}_t f(X_t) 
    - \int_0^t {\mathcal E}_s {\mathcal G}f(X_s)\dd s
    -  \int_0^t  \underline c(X_s) {\mathcal E}_sf(X_s)\dd s
    = 
    {\mathcal E}_t f(X_t) - \int_0^t {\mathcal E}_s \bar{\mathcal A}f(X_s)\dd s
  $$
  is a local martingale.  Plainly, this local martingale 
  remains bounded on any finite time interval,
  and is therefore a true martingale, by \cite[Theorem~I.51]{Pro-sc}.
  We deduce, by taking expectations and using Fubini's theorem,
  that
  $$ 
    \EE_x( {\mathcal E}_t f(X_t) )- f(x)
    = \int_0^t  \EE_x( {\mathcal E}_s \bar{\mathcal{A}} f(X_s) )\,\dd s
  $$
  holds.
  Recalling Lemma \ref{L0}(i), this yields the identity 
  $\bar T_tf(x) =\EE_x( {\mathcal E}_t f(X_t) )$, and we conclude the proof with Lemma \ref{L0}(ii). 
\end{proof}

We mention that the Feynman-Kac representation of the growth-fragmentation
\linebreak
semigroup given in Lemma \ref{Le1} can also be viewed as a `many-to-one formula' 
in the setting of branching particle systems (see, for instance, section 1.3 in \cite{Shi}).
Informally, the growth-fragmentation equation describes the evolution of the intensity of
a stochastic system of branching particles that grow at rate $c$ and split randomly 
according to $k$. In this setting,  the Markov process $(X_t)_{t\geq 0}$ with generator
${\mathcal G}$ arises by following the trajectory of a distinguished
particle in the system, such that after each dislocation event
involving the distinguished particle, the new distinguished particle 
is selected amongst the new particles according to a size-biased sampling.
This particle is referred to as the `tagged fragment' in certain cases
of the growth-fragmentation equation, and we will make this connection
more explicit in \autoref{s:levy}.

In order to study the long time asymptotic behaviour of the 
growth-fragmentation semigroup, we seek to understand
how $ \EE_x[{\mathcal E}_t f(X_t)/X_t]$
behaves as $t\to \infty$.
We shall tackle this issue in the rest of this work by adapting ideas and techniques
of ergodicity for general nonnegative operators,
which have been developed mainly in the discrete time setting in the literature;
see \citet{Num-mc} and \citet{Seneta} for a comprehensive introduction.
We shall rely heavily on the fact that the piecewise deterministic Markov process
$X$ has no positive jumps, and as a consequence,
the probability that the process hits any given single point is positive
(points are `non-polar'.)
This enables us to apply the regenerative property of the process at the 
sequence of  times when it returns  to its starting point.

\section{The spectral radius} 
\label{s:malthus}

Our goal now is to use our knowledge of the Markov process $X$
in order to find the parameter $\rho$ which governs the
decay or growth of solutions to the growth-fragmentation equations.

We introduce
$$H(x)\coloneqq \inf\left\{ t>0: X_t=x\right\},$$
the first hitting time of $x>0$ by $X$.
We stress that, when $X$ starts from $X_0=x$,
$H(x)$ is the first instant (possibly infinite) at which $X$ \emph{returns}
for the first time to $x$. 
Given $x,y>0$, the Laplace transform 
$$
  L_{x,y}(q)
  \coloneqq
  \EE_x\left( \e^{-q H(y)}{\mathcal E}_{H(y)} , H(y)<\infty\right),
  \qquad q\in\RR,
$$
will play a crucial role in our analysis. 
We first state a few elementary facts which will be useful in the sequel. 

Since $X$ is irreducible, we have
$\PP_x(H(y)<\infty)>0$.
Moreover,
${\mathcal E}_{H(y)}>0$ on the event $H(y)<\infty$, 
from which it follows that $L_{x,y}(q)\in(0,\infty]$. The function
$L_{x,y}\from \RR\to (0,\infty]$ is convex, non-increasing,
and right-continuous at the boundary point of its 
domain (by monotone convergence).
Furthermore, we have  $\e^{-qt}{\mathcal E}_t\leq 1$
for every $q> \|\underline c\|_{\infty}$, and then $L_{x,y}(q)<1$;
indeed,  
$$
  \lim_{q\to -\infty} L_{x,y}(q)=\infty
  \quad \text{and} \quad
  \lim_{q\to +\infty} L_{x,y}(q)=0.
$$

The next result is crucial for the identification of the spectral radius. 
\begin{proposition}\label{P-1}%
  Let $q\in\RR$ with $L_{x_0, x_0}(q) < 1$ for some $x_0>0$. Then $L_{x, x}(q) < 1$ for all $x>0$.
\end{proposition}
\begin{proof} 
  Let  $x\not = x_0$
  and observe first from the strong Markov property 
  applied at the first hitting time $H(x)$, that 
  \begin{eqnarr*}
    1&> & \EE_{x_0}({\mathcal E}_{H(x_0)} \e^{-q  H(x_0)}, H(x_0)<\infty) \\
    &\geq& \EE_{x_0}({\mathcal E}_{H(x_0)} \e^{-q  H(x_0)}, H(x)<H(x_0)<\infty) \\
    &= &\EE_{x_0}({\mathcal E}_{H(x)} \e^{-q  H(x)}, H(x)<H(x_0)) 
    \EE_x({\mathcal E}_{H(x_0)} \e^{-q  H(x_0)}, H(x_0)<\infty) \\
    &=& \EE_{x_0}({\mathcal E}_{H(x)} \e^{-q  H(x)} , H(x)<H(x_0)) L_{x,x_0}(q ).
  \end{eqnarr*}
  Since  $\PP_{x_0}(H(x)<H(x_0))>0$, because $X$ is irreductible, this entails that 
  $$
    0< \EE_{x_0}({\mathcal E}_{H(x)} \e^{-q  H(x)} , H(x)<H(x_0))< \infty 
    \quad \text{and} \quad 
    0<L_{x,x_0}(q )<\infty.
  $$

  Next, we work under $\PP_{x_0}$ and write $0=R_0< H(x_0)=R_1< \dotsb$ 
  for the sequence of return times at $x_0$. Using the regeneration 
  at those times, we get
  \begin{eqnarr*}
    L_{x_0,x}(q ) 
    &=& \sum_{n=0}^{\infty} \EE_{x_0}({\mathcal E}_{H(x)} \e^{-q  H(x)},
    R_n < H(x) < R_{n+1}) \\
    &=& \sum_{n=0}^{\infty} \EE_{x_0}({\mathcal E}_{R_n} \e^{-q  R_n}, R_n < H(x))
    \EE_{x_0}({\mathcal E}_{H(x)} \e^{-q  H(x)},  H(x)< R_1) \\
    &=&  \EE_{x_0}({\mathcal E}_{H(x)} \e^{-q  H(x)},  H(x)< H(x_0)) 
    \sum_{n=0}^{\infty} \EE_{x_0}({\mathcal E}_{H(x_0)} \e^{-q  H(x_0)}, H(x_0) < H(x))^n
  \end{eqnarr*}
  Plainly, 
  $$
    \EE_{x_0}({\mathcal E}_{H(x_0)} \e^{-q  H(x_0)}, H(x_0) < H(x))
    \leq  \EE_{x_0}({\mathcal E}_{H(x_0)} \e^{-q  H(x_0)}, H(x_0) <\infty)<  1,
  $$ 
  and summing the geometric series, we get
  \begin{eqnarr*}
    L_{x_0,x}(q )
    &=& \frac{\EE_{x_0}({\mathcal E}_{H(x)} \e^{-q  H(x)},  H(x)< H(x_0))}
    {1-\EE_{x_0}({\mathcal E}_{H(x_0)} \e^{-q  H(x_0)}, H(x_0) < H(x))} \\
    &< & \frac{\EE_{x_0}({\mathcal E}_{H(x)} \e^{-q  H(x)},  H(x)< H(x_0))}
    {\EE_{x_0}({\mathcal E}_{H(x_0)} \e^{-q  H(x_0)}, H(x) < H(x_0)<\infty)} 
    = \frac{1}{L_{x,x_0}(q )}\,,
  \end{eqnarr*}
  where the last equality follows from the strong Markov property 
  applied at time $H(x)$ (and we stress that the ratio in the middle is
  positive and finite.)
  Hence, we have
  \begin{equation} \label{e:ineq}
    L_{x_0,x}(q ) L_{x,x_0}(q )<  1.
  \end{equation}

  We next perform a similar calculation, but now  under $\PP_x$.
  Using regeneration at  return times at $x$ as above, we see that
  $$
    L_{x,x_0}(q )
    = 
    \EE_{x}({\mathcal E}_{H(x_0)} \e^{-q  H(x_0)},  H(x_0)< H(x)) 
    \sum_{n=0}^{\infty} \EE_{x}({\mathcal E}_{H(x)} \e^{-q  H(x)}, H(x) < H(x_0)) ^n.
  $$
  Since we know that $L_{x,x_0}(q )<\infty$, the geometric series above converges, so 
  $$\EE_{x}({\mathcal E}_{H(x)} \e^{-q  H(x)}, H(x) < H(x_0))<1,$$
  and
  $$
    L_{x,x_0}(q ) 
    =
    \frac{ \EE_{x}({\mathcal E}_{H(x_0)} \e^{-q  H(x_0)},  H(x_0)< H(x))}
    {1-\EE_{x}({\mathcal E}_{H(x)} \e^{-q  H(x)}, H(x) < H(x_0))}.
  $$
  Multiplying by $L_{x_0,x}(q )$ and using \eqref{e:ineq}, we deduce that 
  \begin{eqnarr*}
    1-\EE_{x}({\mathcal E}_{H(x)} \e^{-q  H(x)}, H(x) < H(x_0))
    &> &
    \EE_{x}({\mathcal E}_{H(x_0)} \e^{-q  H(x_0)},  H(x_0)< H(x)) L_{x_0,x}(q )\\
    &=& \EE_{x}({\mathcal E}_{H(x)} \e^{-q  H(x)},  H(x_0)< H(x)<\infty),
  \end{eqnarr*}
  where again the last equality is seen from the strong Markov property. It follows  that $\EE_{x}({\mathcal E}_{H(x)} \e^{-q  H(x)}, H(x) <\infty)=L_{x,x}(q)< 1$. 
\end{proof}

We next fix some arbitrary point $x_0>0$, and introduce a fundamental quantity. 
\begin{definition}\label{D1} 
  We call
  $$
    \rho\coloneqq \inf\{q\in\RR: L_{x_0,x_0}(q) < 1\}
  $$
  the \emph{spectral radius} of the growth-fragmentation operator $\mathcal{A}$.
\end{definition}
We stress that Proposition \ref{P-1} shows in particular that the
spectral radius $\rho$ does not depend on the choice of $x_0$. We next justify the terminology by observing that, if $q<\rho$, then 
$$\int_0^{\infty} \e^{-qt} T_t{f}(x) \dd t=\infty$$
for all $x>0$ and all continuous functions 
$f\colon (0,\infty)\to \RR_+$ with $f\not\equiv 0$, 
whereas, if $q>\rho$, then  there exists a function $f$
which is everywhere positive, and such that
$$\int_0^{\infty} \e^{-qt} T_t{f}(x) \dd t<\infty$$
for all $x>0$.
The following result actually provides a slightly stronger statement.

\begin{proposition}\label{P-3}%
  Let $q\in\RR$.
  \begin{enumerate}
    \item
    If $L_{x,x}(q)\geq 1$, then for every  $f\colon(0,\infty)\to [0,\infty)$ continuous with
    $f\not \equiv 0$, we have
    $$\int_0^{\infty} \e^{-qt} T_t{f}(x) \dd t=\infty.$$
    
    \item 
    If $L_{x,x}(q)< 1$, then there exists a function 
    $f\colon (0,\infty)\to (0,\infty]$ with
    $$\lim_{t \to 0} \e^{-qt} T_t{f}(x) =0.$$
  \end{enumerate}
\end{proposition} 
\begin{proof}
  (i) Recall from Lemma \ref{Le1} that
  $$\int_0^{\infty} \e^{-qt} T_t{f}(x)\dd t 
    = x \EE_x\left(\int_0^{\infty} \e^{-q t } {\mathcal E}_t \underline f(X_t)\dd t\right).
  $$
  Decomposing $[0,\infty)$ according to the return times of $X$ at its 
  starting point and applying the regeneration property just as in the 
  proof of Proposition \ref{P-1}, we easily find
  that the quantity above equals 
  $$
    x \EE_x\left(\int_0^{H(x)} \e^{-q t } {\mathcal E}_t \underline f(X_t)\dd t\right)
    \sum_{n=0}^{\infty}
    \EE_x\left( \e^{-q H(x)}{\mathcal E}_{H(x)}, H(x)<\infty\right)^n.
  $$
  Now the first term above is positive since $f\geq 0$, $f\not \equiv 0$ and $X$ is irreducible, and the series diverges because
  $\EE_x\left( \e^{-q H(x)}{\mathcal E}_{H(x)}, H(x)<\infty\right)
  = L_{x,x}(q) \ge 1$. 
  
  (ii) We take $f(y)= y L_{y,x}(q)$ and observe from the Markov property and
  Lemma \ref{Le1} that then
  $$
    \e^{-qt} T_tf(x)
    =
    x \EE_x\left( \e^{-q R(t) } {\mathcal E}_{R(t)}, R(t)<\infty\right),
  $$
  where $R(t)$ denotes the first return time of $X$ to $x$ after time $t$.
  We use the notation $\theta_\cdot$ for the usual shift operator;
  that is, $(X_s, s\ge 0) \circ \theta_t = (X_{s+t}, s\ge 0)$.
  As before, we denote the sequence of return times of $X$ to
  its starting point by
  $R_0 = 0 < R_1 < \dotsb$.
  With this notation, we have that $R(t)=R_{n+1}$ if and only if $R_n\leq t$ and 
  $H(x)\circ \theta_{R_n}> t-R_n$. 
  Regeneration at the return times then enables us to express 
  $\e^{-qt} T_tf(x)$ as  
  \begin{align*}
  &  x \sum_{n=0}^{\infty} \int_{[0,t]} 
    \EE_x\left( \e^{-q R_n}{\mathcal E}_{R_n}, R_n\in \dd s\right)
    \EE_x\left( \e^{-q H(x)}{\mathcal E}_{H(x)}, t-s <H(x)<\infty\right)\\
   & {} \eqqcolon x\int_{[0,t]}U^q(x,\dd s) \varphi_x(t-s),
%&\leq  &  \sum_{n=0}^{\infty}  \EE_x\left( \e^{-q R_n}{\mathcal E}_{R_n}, R_n<\infty \right) L_{x,x}(q).
  \end{align*}

  On the one hand, we observe, again by regeneration, 
  that the total mass of the measure $U^q(x,\cdot) $ is given by
  $$
    U^q(x,[0,\infty))
    =
    \sum_{n=0}^{\infty} \EE_x\left( \e^{-q R_n}{\mathcal E}_{R_n}, R_n<\infty \right)
    = 
    \sum_{n=0}^{\infty} L_{x,x}(q)^n<\infty,
  $$
  On the other hand, since
  $$\EE_x\left( \e^{-q H(x)}{\mathcal E}_{H(x)}, H(x)<\infty\right)=L_{x,x}(q)<\infty,$$ 
  we know that $\lim_{t\to \infty} \varphi_x(t) = 0$. Hence, for every $s\geq 0$, we have $\lim_{t\to \infty} \varphi_x(t-s) = 0$,
  and since $0\leq \varphi_x(t-s) \leq L_{x,x}(q)$ and the measure $U^q(x,\cdot)$ is finite, 
 we can  conclude the proof by dominated convergence. 
\end{proof} 

We now conclude this section by describing
the following elementary bounds for the spectral radius. 
\begin{proposition}\label{P-2}
  \begin{enumerate}
    \item It always holds that  $\rho \leq \|\underline c\|_{\infty}$. 

    \item It holds that $\rho > 0$ whenever $X$ is recurrent; 
    furthermore, if $X$ is positive recurrent with stationary law $\pi$, then 
    $$\rho \geq \ip{\pi}{\underline c}.$$
  \end{enumerate}
\end{proposition}
\begin{proof}
  (i) This follows from the elementary observations preceding Definition \ref{D1}. 

  (ii)
  If $X$ is recurrent, then  $\PP_x(H(x)<\infty)=1$ and 
  $L_{x,x}(0)=\EE_x({\mathcal E}_{H(x)})\in(1,\infty]$. 
  This forces $\rho > 0$, since  $L_{x,x}(\rho)\leq 1$ 
  by right-continuity of $L_{x,x}$.  
  Furthermore, we may apply the regeneration property at 
  the $n$-th return time of $X$ to $x_0$, say $R_n$, and observe that 
  $$ \EE_{x_0}\left(   \e^{-q R_n}{\mathcal E}_{R_n}\right)= L_{x_0, x_0}(q)^n$$
  converges to $0$ as $n\to \infty$ for every $q>\rho$. 
  By the ergodic theorem for positive recurrent Markov processes
  \cite[Theorem 20.20]{Kal},
  $$\ln {\mathcal E}_{R_n}=\int_0^{R_n}{\underline c}(X_s)\dd s 
    \sim  \ip{\pi}{\underline c}R_n 
    \quad \text{ as } n\to \infty, \quad\text{$\PP_{x_0}$-a.s.,}
  $$
  and we then see from Fatou's Lemma that 
  $\lim_{n\to \infty}  \EE_{x_0}\left(   \e^{-q R_n}{\mathcal E}_{R_n}\right)=\infty$,
  as long as $q< \ip{\pi}{\underline c}$. This entails our last claim.
\end{proof}

\section{A martingale multiplicative functional}
\label{s:mult}

In short, the purpose of this section is to construct a remarkable martingale 
which we will then use to transform the Markov process $X$. 
We shall obtain a recurrent Markov process $Y$ which in turn will enable us to reduce
the analysis of the asymptotic behaviour of $T_t$ to results from ergodic theory.
This  requires
the following assumption to hold: 
\begin{equation} \label{e:rt0} 
  \hbox{  $L_{x_0,x_0}(\rho)=1$.}
\end{equation} 
Note that, by the right-continuity of $L_{x, x}$, we always have
$L_{x_0,x_0}(\rho)\leq 1$. 
 
We start with some simple observations relating \eqref{e:rt0}
to the value of $L_{x_0,x_0}$ at the left endpoint of its domain. 

\begin{lemma}\label{Le2}
  Define $q_*\coloneqq \inf\{q\in\RR: L_{x_0,x_0}(q)<\infty\}$.  Then:
  \begin{enumerate}  
    \item Condition \eqref{e:rt0} holds if and only if $L_{x_0,x_0}(q_*)\in[1,\infty]$.
    \item If $L_{x_0,x_0}(q_*)\in (1,\infty]$, then $L_{x_0,x_0}$ possesses a finite right-derivative at $\rho$ and 
    $$\EE_{x_0}\left( H(x_0) \e^{-\rho H(x_0)} {\mathcal E}_{H(x_0)} , H(x_0)<\infty\right)
      = -L'_{x_0,x_0}(\rho)<\infty.$$
  \end{enumerate}
\end{lemma}
\begin{proof}
Recall that $q_*\leq \|\underline c\|_{\infty} $ and that $L_{x_0, x_0}$  is convex and decreasing. We have  
$$\lim_{q\to \infty}L_{x_0, x_0}(q)=0 \quad \hbox{and} \quad \lim_{q\to q_*+}L_{x_0, x_0}(q)=L_{x_0, x_0}(q_*)$$
by dominated convergence for the first limit, and  by monotone convergence for the second. 
This yields our first claim. For the second, it suffices to observe that if $L_{x_0,x_0}(q_*)>1$, then $\rho > q_*$ and thus, by convexity,  the right derivative of $L_{x_0, x_0}$ at $\rho$ is finite. 
\end{proof}

We assume throughout the rest of this section that  \eqref{e:rt0} holds, 
and describe some remarkable properties of the function $(x,y)\mapsto L_{x,y}(\rho)$ 
which follow from this assumption.

\begin{lemma} \label{Le3} Assume that  \eqref{e:rt0} holds for some $x_0>0$. Then
  \begin{enumerate}
    \item $L_{x,x}(\rho)=1$ for all $x>0$,
    i.e., \eqref{e:rt0} actually holds with $x_0$ replaced by any $x>0$. 

    \item For all $x,y>0$, we have 
    $$L_{x,y}(\rho)L_{y,x}(\rho)=1.$$

    \item For all $x,y,z>0$, there is the identity 
    $$L_{x,y}(\rho)L_{y,z}(\rho)=L_{x,z}(\rho).$$
  \end{enumerate}
\end{lemma} 

\begin{proof} (i) Indeed, the strict inequality  $L_{x,x}(\rho)< 1$ is ruled out by Proposition \ref{P-1}. On the other hand, we always have $L_{x,x}(\rho)\leq 1$ by the right-continuity of $L_{x,x}$, since, again by Proposition \ref{P-1}, $\rho=\inf\{q\in\RR: L_{x,x}(q)<1\}$. 

(ii)  Using the regeneration at return times at $x$ just as in the proof of Proposition \ref{P-1}, we easily get
\begin{eqnarray*}
L_{x,y}(\rho )&=& \frac{\EE_{x}({\mathcal E}_{H(y)} \e^{-\rho  H(y)},  H(y)< H(x))}{1-\EE_{x}({\mathcal E}_{H(x)} \e^{-\rho  H(x)}, H(x) < H(y))}\\
&=& \frac{\EE_{x}({\mathcal E}_{H(y)} \e^{-\rho  H(y)},  H(y)< H(x))}{\EE_{x}({\mathcal E}_{H(x)} \e^{-\rho  H(x)}, H(y) < H(x)<\infty)}  = \frac{1}{L_{y,x}(\rho )}\,,
\end{eqnarray*} where the last equality follows from the strong Markov property applied at time $H(y)$. 

(iii) Finally, recall that $X$ has no positive jumps, so for every $x<y<z$, we have $H(y)<H(z)$, $\PP_x$-a.s. on the event $H(z)<\infty$, and the strong Markov property readily yields (iii) in that case. Using (ii), it is then easy to deduce that (iii) holds in full generality, no matter the relative positions of $x,y$ and $z$. 
\end{proof}

\begin{corollary} \label{C5}
  The function $(x,y)\mapsto L_{x,y}(\rho)$ is continuous
  on $(0,\infty)$ in each of the variables $x$ and $y$.
\end{corollary}
\begin{proof}
  We only need to check that $\lim_{y\to x} L_{x,y}(\rho)=1$. If this holds, then
  Lemma \ref{Le3}(iii) then entails the continuity of $z\mapsto L_{x,z}(\rho)$
  and we can
  conclude from  Lemma \ref{Le3}(ii) that $x\mapsto L_{x,y}(\rho)$ is also continuous. 

  In this direction, observe first  that $X$ has no positive jumps and follows a
  positive flow velocity
  between its jump times. Thus, $\PP_x$-a.s., on the event $H(x)<\infty$,
  there exists a unique instant $J\in(0,H(x))$ such that $X_t>x$ for $0< t < J$ and
  $X_t<x$ for $J< t < H(x)$. Further, $X$ is continuous at times $0$ and $H(x)$. 
  In particular, we have $\PP_x$-a.s. that $\lim_{y\to x+}H(y)=0$ whereas
  $\lim_{y\to x-}H(y)=H(x)$,   and actually, the following limits
  \begin{eqnarr*}
    \lim_{y\to x+} \e^{-\rho H(y)}{\mathcal E}_{H(y)}{\bf 1}_{\{H(y)<\infty\}} &=& 1, \\
    \lim_{y\to x-} \e^{-\rho H(y)}{\mathcal E}_{H(y)}{\bf 1}_{\{H(y)<\infty\}} 
    &=& \e^{-\rho H(x)}{\mathcal E}_{H(x)}{\bf 1}_{\{H(x)<\infty\}},
  \end{eqnarr*}
  hold $\PP_x$-a.s.
  We observe that the $\PP_x$-expectation of the last quantity 
  is $L_{x,x}(\rho)=1$ (by Lemma \ref{Le3}(i)), and deduce from Fatou's lemma that
  $$\liminf_{y\to x} L_{x,y}(\rho) \geq 1.$$

  On the other hand, recall that $K(x)=\int_0^x \bar k(x,y)\dd y$ is the total rate of
  jumps at location $x$. 
  An easy consequence of the fact that  $X$ follows the flow velocity given by
  $\dd x(t)=c(x(t))\dd t$
  between its jumps, is that the probability under $\PP_y$ of the event $\Lambda_{x}$ that  $X$ has no jump before hitting $x>y$ is given by
  $$\PP_y(\Lambda_{x})= \exp\biggl(-\int_y^x\frac{K(z)}{c(z)}\dd z\biggr),$$
  a quantity which converges to $1$ as $y\to x-$. 
  Moreover, the time $h(x)$ at which the flow velocity started from $y$ reaches the
  point $x$ is given by
   \[h(y,x) =    \int_y^x \frac{1}{c(s)}\,\dd s , \]
  a quantity which converges to $0$ as $y\to x-$.
  Using $L_{y,x}(\rho)\geq \e^{-\rho h(y,x)}\PP_y(\Lambda_x)$, we deduce that
  $\liminf_{y\to x-}L_{y,x}(\rho)\geq 1$, and then, thanks to Lemma \ref{Le3}(ii) that 
  $$\limsup_{y\to x-}L_{x,y}(\rho)\leq 1,$$
  from which it follows that $\lim_{y\to x-} L_{x,y}(\rho) = 1$ and, by the Lemma \ref{Le3}(iii),
  that also $\lim_{y\to x-} L_{y,x}(\rho)=1$.
  
  Finally, working now under $\PP_x$ and, just as above,
  denoting by $\Lambda_y$
  the event that
  $X$ makes no jumps before hitting $y$,
  we obtain by monotone convergence that
  $$\lim_{y\to x+} \EE_x\bigl[\e^{-\rho H(x)} {\mathcal E}_{H(x)}
    \Ind_{\Lambda_y}\, \Indic{H(x)<\infty}\bigr]
    =
    L_{x,x}(\rho) = 1.$$
  If we write $h(x,y)$ for the hitting time of $y$ by the flow velocity $x(\cdot)$ started from $x$,
  and observe that $\int_0^{h(x,y)} \underline c(x(s))\dd s=\ln(y/x)$, 
  we obtain by the Markov property at time $h(x,y)$ that 
  $$
    \EE_x\bigl[\e^{-\rho H(x)} {\mathcal E}_{H(x)}
    \Ind_{\Lambda_y} \Indic{H(x)<\infty}\bigr] =\e^{-\rho h(x,y)} \frac{y}{x} L_{y,x}(\rho).
  $$
  Since $\lim_{y\to x+} h(x,y)=0$, we conclude, using again Lemma \ref{Le3}(ii)
  for the second equality below, that 
  $$\lim_{y\to x+}  L_{y,x}(\rho) = 1 = \lim_{y\to x+}  L_{x,y}(\rho),$$
  and the proof is complete. 
\end{proof} 

Once again, we recall our standing assumption that \eqref{e:rt0} holds.
The following function will be crucial for our analysis:
$$\ell(x) = L_{x,x_0}(\rho)\,,\qquad x>0.$$ 
Note from Lemma~\ref{Le3}(iii) that, for any $y_0>0$ and $x>0$,
$L_{x,y_0}(\rho) = \ell(x) L_{x_0,y_0}(\rho)$,
and so replacing
$x_0$ by $y_0$ would only affect the function $\ell$ by a constant
factor. Further, we know from Corollary \ref{C5} that 
$\ell$ is continuous and positive on $(0,\infty)$; 
in particular, it remains bounded away from $0$ and from $\infty$ on 
compact subsets of $(0,\infty)$.

% the functions $y\mapsto L_{y,x_0}(\rho)$ and $y\mapsto L_{y,x}(\rho)$ are proportional,

We then introduce the multiplicative functional 
$$
  {\mathcal M}_t\coloneqq \e^{-\rho t} {\mathcal E}_t  \frac{\ell(X_t)}{\ell(X_0)}\,,
  \qquad t\geq 0.
$$
The qualifier \emph{multiplicative} stems from the identity 
${\mathcal M}_{t+s}= {\mathcal M}_s\circ \theta_t \times {\mathcal M}_t $, 
where $\theta_t$ denotes the usual shift operator. 
Our strategy in the sequel shall be to make a change of measure with
respect to this multiplicative functional.
The following result is
therefore very important for our goal.

\begin{theorem}\label{Th1}
   For every $x>0$, the multiplicative functional
  $({\mathcal M}_t)_{t\geq 0}$ is a $\PP_x$-martingale
  with respect to the natural filtration $(\mathcal{F}_t)_{t\ge 0}$ of $X$.
\end{theorem}

\begin{proof}
  Without loss of generality, we shall  work under $\PP_{x_0}$. 
  We also define the random variables
  $R_0=0<R_1\coloneqq H(x_0)<R_2< \cdots$ to be the sequence 
  of return times to the point $x_0$, and recall from the regenerative property
  at these return times that  for every $n\geq 0$, conditionally on 
  $R_n<\infty$, the ratio  
  $$\frac{ \e^{-\rho R_{n+1}}{\mathcal E}_{R_{n+1}} }{ \e^{-\rho R_n} {\mathcal E}_{R_n}}= \exp\left(\int_{R_n}^{R_{n+1}} (\underline c(X_s)-\rho)\dd s\right) $$ 
is independent of ${\mathcal F}_{R_n}$ and has the same law as ${\mathcal E}_{H(x_0)} \e^{-\rho H(x_0)}$ under $\PP_{x_0}$.
We see from  \eqref{e:rt0} that $\EE_{x_0} \left( {\mathcal E}_{R_n} \e^{-\rho R_n}, R_n<\infty\right) =1$  for every $n\geq 0$, and it then 
 follows  from the Markov property  that there is the identity
\begin{eqnarray*}
\EE_{x_0} \left( {\mathcal M}_{R_n} , R_n<\infty \mid {\mathcal F}_t\right)&= &\EE_{x_0} \left(  \e^{-\rho R_n} {\mathcal E}_{R_n}, R_n<\infty \mid {\mathcal F}_t\right)\\
& =& \e^{-\rho (t\wedge R_n)} {\mathcal E}_{t\wedge R_n}\ell(X_{t\wedge R_n}) \\
&=& {\mathcal M}_{t\wedge R_n}.
\end{eqnarray*}
As a consequence, the stopped process $({\mathcal M}_{t\wedge R_n})_{t\geq 0}$ is a martingale.

Further, if we introduce the tilted probability measure 
$$\QQ^n=  {\bf 1}_{R_n<\infty}  \e^{-\rho R_n}{\mathcal E}_{R_n} \PP_{x_0} =  {\bf 1}_{R_n<\infty}{\mathcal M}_{R_n} \PP_{x_0},$$ then we see by the regeneration property at the return times and the fact that ${\mathcal M}$ is a multiplicative functional, that under
$\QQ^n$, the variables $R_1, R_2-R_1, \ldots, R_n-R_{n-1}$ are i.i.d. with law
$$\QQ^n(H(x_0)\in \dd s)= \PP_{x_0} (\e^{-\rho H(x_0)} {\mathcal E}_{H(x_0)} , H(x_0)\in \dd s)\,,\qquad s\in(0,\infty).$$
We stress that this distribution does not depend on $n$, and in particular, for every $t>0$, we have
$$ \EE_{x_0} \left( {\mathcal M}_{R_n} , R_n\leq t\right)  = \QQ^n(R_n\leq t) \longrightarrow 0 \hbox{ as }n\to \infty.$$
To complete the proof, it now suffices to write for every $t\geq s \geq 0$
\begin{eqnarray*}
{\mathcal M}_{s\wedge R_n}&=& \EE_{x_0}({\mathcal M}_{t\wedge R_n}\mid {\mathcal F}_s)\\
&=& \EE_{x_0}({\mathcal M}_{t}, R_n>t\mid {\mathcal F}_s)+  \EE_{x_0}({\mathcal M}_{R_n}, R_n\leq t \mid {\mathcal F}_s),
\end{eqnarray*}
and we conclude by letting $n\to \infty$ that ${\mathcal M}_s=\EE_{x_0}({\mathcal M}_t\mid {\mathcal F}_s)$. 
 \end{proof}
We point out that the continuity of $\ell$ (which is a special case of Corollary \ref{C5}) could also be established
from Theorem~\ref{Th1} and classical regularity properties of martingales.
We conclude this section by the following  easy consequence of Theorem~\ref{Th1}.
Under rather mild assumptions, we identify  
the function $\bar \ell(x)=x\ell(x)$ as an eigenfunction of
the growth-fragmentation operator ${\mathcal A}$, with eigenvalue given by the spectral radius $\rho$.

\begin{corollary}\label{Co1} \begin{enumerate}
\item The function $\ell$ belongs to the extended domain of the infinitesimal generator ${\mathcal G}$ of $X$ with ${\mathcal G}\ell=(\rho-\underline c)\ell$,  in the sense that the process
\begin{equation}\label{e:h-martingale}
  \ell(X_t) -\int_0^t \left(\rho - \underline c(X_s) \right) \ell(X_s)\dd s
\end{equation}
is a  martingale under $\PP_x$ for every $x>0$.

\item If $\ell$ is bounded on $(0,\infty)$, then 
$\bar \ell \in {\mathcal D}({\mathcal A})$ and ${\mathcal A}\bar \ell=\rho \bar \ell$. 
\end{enumerate}
\end{corollary}

\begin{proof}[Proof of \autoref{Co1}] (i)
  Indeed, it suffices to write 
  $$
    \ell(X_t)
    = 
    \ell(x) {\mathcal M}_t 
    \e^{\rho t} \exp\left( - \int_0^t \underline c (X_s)\dd s\right)
  $$
 and apply stochastic integration by parts.
 We obtain
 $$  \ell(X_t) =\ell(x) + \ell(x) \int_0^t \e^{\rho s} {\mathcal E}_s \dd{\mathcal M}_s + \int_0^t \left(\rho - \underline c(X_s) \right) \ell(X_s)\dd s.$$
 On the time interval $[0,t]$, the integrand $\e^{\rho s} {\mathcal E}_s$ in the stochastic integral is bounded by a constant, and this entails that the process in \eqref{e:h-martingale}
 is a martingale, by \cite[Theorem~I.51]{Pro-sc}.
 
 (ii)  Recall that we already know that $\ell$ is continuous, so if further $\ell$ is bounded, then $\ell\in{\mathcal C}_b$. Then also  $(\rho-\underline c)\ell \in{\mathcal C}_b$,
 and, by taking expectations in 
 \eqref{e:h-martingale} and using the Feller property of $X$,
 (i) entails that $\ell$ belongs to the domain of the infinitesimal generator ${\mathcal G}$, that is $\ell \in {\mathcal D}(\bar {\mathcal A})$ or equivalently 
 $\bar \ell \in {\mathcal D}( {\mathcal A})$,  with
 ${\mathcal G}\ell=(\rho-\underline c)\ell$. Since ${\mathcal G}f(x)=x^{-1}{\mathcal A}\bar f(x)-\underline c(x) f(x)$, we conclude that ${\mathcal A}(\bar \ell)= \rho \bar \ell$. 
 \end{proof}

In order to apply Corollary \ref{Co1}(ii), we need explicit conditions ensuring that $\ell$ is bounded,
and in this direction we record the following result.

\begin{lemma} \label{Le:lb} Assume  that
 $$\limsup_{x\to 0+}\underline c(x)<\rho\quad \hbox{and} \quad \limsup_{x\to \infty}\underline c(x)<\rho.$$
  Then $\ell \in{\mathcal C}_b$. 
\end{lemma}

\begin{proof}
Under the assumptions of the statement, there exists $\rho'<\rho$ such that 
the set
$\{x>0: \underline c(x)\geq \rho'\}$ is a compact subset of $(0,\infty)$;
assume that it is contained in $[a,b]$, for some $0<a<x_0<b$.
Now, since $\ell$ is continuous, it is certainly bounded on $[a,b]$.
Moreover, if $0<x<a$, then $\e^{-\rho H(a)}{\mathcal E}_{H(a)} \leq \e^{-(\rho -\rho')H(a)} \leq 1$. So $L_{x,a}(\rho)\leq 1$, and by Lemma \ref{Le3}(iii), $\ell$ remains bounded on $(0,a)$. 

Similarly, if now $x>b$ and  $H(a,b)\coloneqq \inf\{t>0: X_t\in[a,b]\}$ denotes the first entrance time in $[a,b]$,  then again
 $\e^{-\rho H(a,b)}{\mathcal E}_{H(a,b)} \leq \e^{-(\rho -\rho')H(a)} \leq 1$. By the strong Markov property applied at time $H(a,b)$, we conclude that 
 $\ell(x)\leq  \max_{[a,b]}\ell$, so $\ell$ remains bounded on $(b,\infty)$. 
 \end{proof}

\section{Applying ergodic theory for Markov processes} 
\label{s:ergodic}

We still assume that \eqref{e:rt0} holds throughout this section.
Having established the existence of the martingale multiplicative
functional ${\mathcal M}$,
we use this to `tilt' the  initial probability measure $\PP_x$.
In other words, we introduce a new probability measure $\QQ_x$,
defined by the following formula for every $A\in \mathcal{F}_t$:
\[ \QQ_x(A) = \EE_x[ \Ind_A \mathcal{M}_t] .\]
Since $\PP_x$ is a probability law on the space of c\`adl\`ag paths,
the same holds for $\QQ_x$; and it is convenient to denote
by $Y=(Y_t)_{t\geq 0}$ a process with distribution $\QQ_x$. 
For clarity, let us point out that its finite-dimensional
distributions are given as follows. Let $0\le t_1<\dotsb<t_n\le t$,
and $F \from \RR^n\to\RR$. Then
\[ \QQ_x[ F(Y_{t_1},\dotsc,Y_{t_n})]
  = \EE_x[ \mathcal{M}_t F(X_{t_1},\dotsc,X_{t_n})], \qquad x > 0.
\]
(Note that, whenever it will not cause confusion, we will
use $\QQ_x$ not just for the probability measure, but also
for expectations under this measure.)
In fact, $Y$ is not just a stochastic process, but a Markov process,
and we can specify its distribution in detail, as follows.

\begin{lemma}\label{Le4} Let $x>0$.
  \begin{enumerate}
    \item\label{Le4:mp}
    Under the measure $\QQ_x$, $Y = (Y_t)_{t\ge 0}$ is a strong Markov
    process. The domain of its extended infinitesimal generator ${\mathcal G}_Y$
    contains  ${\mathcal D}_{\ell}({\mathcal G})\coloneqq \{g: g\ell \in {\mathcal D}({\mathcal G})\}$,
    and is given by 
    \begin{equation} \label{e:gY}
      {\mathcal G}_Yg(x) 
      = \frac{1}{\ell(x)} {\mathcal G}(g\ell)(x) +(\underline c(x)- \rho) g(x) 
    \end{equation}
    in the sense that, for every $x>0$ and $g\in {\mathcal D}_{\ell}({\mathcal G})$, 
    \begin{equation} \label{e:qmart}
      g(Y_t)-\int_0^t {\mathcal G_Y}g(Y_s) \, \dd s
      \quad \hbox{is a local martingale under $\QQ_x$.}
    \end{equation}
    Its semigroup 
    $(T^Y_t)_{t\geq 0}$, defined on the Banach space 
    \[ {\mathcal C}^{\ell}_b
    \coloneqq \{g: (0,\infty)\to (0,\infty): g\ell \in{\mathcal C}_b\} \]
    with norm $\|g\|=\|g\ell\|_{\infty}$, is
    given by 
    \[ T^Y_tg(x) \coloneqq \QQ_x[ g(Y_t) ]
      =\EE_x({\mathcal M}_t g(X_t))
      = \frac{1}{\ell(x)}
      \EE_x\left( \e^{-\rho t}{\mathcal E}_t \ell(X_t) g(X_t)\right).
    \]
    \item\label{Le4:rec}
    $Y$ is point recurrent.
  \end{enumerate}
\end{lemma}
\begin{proof}
  \ref{Le4:mp}
  It is well-known that transformations based on 
  multiplicative functionals preserve the (strong) Markov property;
  we refer to \cite[\S III.19]{RW1} for a readable account
  of a slightly simpler case,
  or \cite[\S 62]{Sha-mp} for a technical discussion.
  We can thus view $\QQ_x$ as the law of a Markov process 
  $(Y_t)_{t\geq 0}$ with values in $(0,\infty)$,
  whose semigroup is given by $T_t^Y$.
  
  We now prove \eqref{e:qmart} for every $x>0$.
  Indeed, we know that  $f(X_t)-\int_0^t {\mathcal G} f(X_s)\dd s$ 
  is a $\PP_x$-martingale, so by stochastic calculus,
  $$
    \e^{-\rho t} {\mathcal E}_t f(X_t)
    - \int_0^t \e^{-\rho s} {\mathcal E}_s
    \left( {\mathcal G}f(X_s) +(\underline c(X_s)-\rho)f(X_s)\right) \dd s
  $$
  is  a $\PP_x$-local martingale. Multiplying by $\ell(x)$, this shows that
  $$
    {\mathcal M}_t g(X_t) -\int_0^t \frac{ {\mathcal M}_s }{\ell(X_s)}
    \left( {\mathcal G}f(X_s) +(\underline c(X_s)-\rho)f(X_s)\right) \dd s
  $$
  is a $\PP_x$-local martingale. 
  Further, since ${\mathcal M}$ is a $\PP_x$-martingale,
  stochastic integration by parts shows that for every locally bounded function $h$, 
  $$
    {\mathcal M}_t \int_0^t h(X_s)\dd s -\int_0^t  {\mathcal M}_s h(X_s) \dd s
  $$
  is again $\PP_x$-local martingale. Putting the pieces together, we get that
  $$
    {\mathcal M}_t \left(g(X_t) 
    - \int_0^t
    \frac{ {\mathcal G}f(X_s) + (\underline c(X_s)-\rho)f(X_s)}
    {\ell(X_s)} \dd s \right)
  $$
  is a $\PP_x$-local martingale, that is, equivalently, \eqref{e:qmart} holds. 

  \ref{Le4:rec}
  Write $H_Y(x)=\inf\{t>0: Y_t=x\}$ for first hitting time of $x$ by the  process $Y$. Then:
  \begin{eqnarr*}
    \QQ_x(H_Y(x)<\infty)&=& \lim_{t\to \infty} \QQ_x(H_Y(x)\leq t)\\
    &=&  \lim_{t\to \infty} \EE_x\left( {\mathcal M}_t, H(x)\leq t\right)\\
    &=& \lim_{t\to \infty} \EE_x\left( {\mathcal M}_{H(x)}, H(x)\leq t\right) \\
    &=& \EE_x\left( {\mathcal M}_{H(x)}, H(x)< \infty\right) =1,
  \end{eqnarr*}
  where at the third equality, we used the optional sampling theorem
  \cite[Theorem II.77.5]{RW1} for the martingale ${\mathcal M}$.
\end{proof}

We next specify classical formulas for invariant measures and stationary distributions of point-recurrent Markov processes, in the case of the process $Y$.

\begin{corollary}\label{Co2}
  \begin{enumerate}
    \item The occupation measure $m_0$ of the excursion of $Y$ away from $x_0$ defined by 
    $$\ip{m_0}{f} \coloneqq \QQ_{x_0}\left(\int_0^{H_Y(x_0)} f(Y_s)\dd s\right),
      \qquad f\in{\mathcal C}_c,$$
    where $H_Y(x)=\inf\{t>0: Y_t=x\}$ denotes the first hitting time of 
    $x$ by the  process $Y$, 
    is  the unique (up on a constant factor) invariant measure for $Y$.
    Further $m_0$  is absolutely continuous with respect to the Lebesgue measure,
    with a locally integrable and  everywhere positive density given by
    $$\frac{q(x_0,y)}{c(y) q(y,x_0)} \,, \qquad y>0,$$
    where $q(x,y)\coloneqq \QQ_x(H_Y(y)<H_Y(x))$. 

    \item $(Y_t)_{t\geq 0}$ is positive recurrent if and only if
    the function $L_{x,x}$ has a finite right-derivative at $\rho$,
    that is,
    \begin{equation}\label{e:+R}
      -L'_{x,x}(\rho) =\EE_{x}\left( H(x) \e^{-\rho H(x)} {\mathcal E}_{H(x)} , H(x)<\infty\right)<\infty
    \end{equation}
    for some (and then all) $x>0$. In that case, its  stationary law, 
    that is $m_0$ normalized to be a probability measure, has the density
    $$\frac{1}{c(y) |L'_{y,y}(\rho)|},\qquad y>0. $$
  \end{enumerate}
\end{corollary}
We recall that Lemma \ref{Le2}(ii) provides a sufficient condition in 
terms of the function $L_{x_0,x_0}$ that ensures that \eqref{e:+R} holds. 

\begin{proof} 
  (i)
  Indeed, it is well-known that  the mean occupation
  measure of an excursion of $Y$ yields an invariant measure of $Y$; see, for instance,
  Getoor  \cite[\S 7]{Getoor}.
  Moreover, since $Y$ is irreducible and recurrent, its invariant measure is unique
  up to multiplication by a constant; see \cite[Theorem 1]{Har-inv}.
  
  The absolute continuity assertion is deduced  
  from the fact that $Y$ is piecewise deterministic, 
  and more precisely follows the deterministic flow  $\dd y(t)= c(y(t))\dd t$ 
  between its jump times. Specifically, one has then 
  $$\int_0^{H_Y(x_0)} f(Y_s)\dd s=\int_0^{\infty} f(y) \frac{N(y)}{c(y)}\dd y,$$
  where $N(y)={\rm Card}\{t\in[0, H_Y(x_0)): Y_t=y\}$ is the number of visits to $y$ of the excursion of $Y$ away from $x_0$. In the notation of the statement, it is readily checked that 
  $\QQ_{x_0}(N(y))= q(x_0,y)/q(y,x_0)$, and this yields the expression for the density.

(ii) Using the formula for $m_0$, the probability tilting, and the martingale property of ${\mathcal M}$,  we have
\begin{eqnarray*}
\ip{m_0}{\bf 1}&=& \int_0^{\infty} (1-\QQ_{x_0}(H_Y(x_0)\leq  t))\dd t \\
&=&  \int_0^{\infty}  (1-\EE_{x_0}\left( {\mathcal M}_t, H(x_0)\leq t\right)) \dd t \\
&=&  \int_0^{\infty}\left (1-  \EE_{x_0}\left( {\mathcal M}_{H(x_0)}, H(x_0) \leq t  \right)\right ) \dd t \\
&=&  \int_0^{\infty} \EE_{x_0}\left( {\mathcal M}_{H(x_0)}, t<H(x_0) <\infty  \right) \dd t \\
&=& \EE_{x_0}\left(H(x_0) {\mathcal M}_{H(x_0)}, H(x_0)<\infty\right).
\end{eqnarray*}
This proves the first assertion (eventually replacing $x_0$ by $x$, which only affects the invariant measure by a constant factor). 

The second assertion follows then from uniqueness of the stationary distribution and the fact that the maps $y\mapsto q(x_0,y)$ and $y\mapsto q(y,x_0)$ both have limit $1$ as $y$ tends to $x_0$. This claim can be proved much in the same way as Corollary \ref{C5}, and the full details are left to the reader. 
\end{proof}
We also point at  the following alternative expressions for the occupation measure $m_0$: 
\begin{eqnarray*}
\ip{m_0}{f} &=&\EE_{x_0}\left(\e^{-\rho H(x_0)} {\mathcal E}_{H(x_0)}\int_0^{H(x_0)} f(X_s)\dd s, H(x_0)<\infty\right)\\
& =&\EE_{x_0}\left(\int_0^{H(x_0)}  \e^{-\rho s}\mathcal{E}_s \ell(X_s)f(X_s) \dd s, H(x_0)<\infty\right),
\end{eqnarray*}
which follow readily from the probability tilting 
and the martingale property of ${\mathcal M}$.

We now state our main result about the asymptotic
behaviour of growth-fragmentation semigroups. 

\begin{theorem}\label{Th2} 
  Assume that \eqref{e:rt0} and \eqref{e:+R} hold,
  so that $Y$ is positive recurrent. Let
  $$\nu(\dd y) \coloneqq \frac{m_0(\dd y) }{\bar \ell(y) \ip{m_0}{\bf 1}} 
    = \frac{\dd y }{ c(y) \bar \ell(y) |L'_{y,y}(\rho)|},\qquad y>0.$$
  Then for every continuous function $f$ with compact support, we have
  $$\lim_{t\to \infty}\e^{-\rho t} T_tf(x)
    = \bar \ell(x) \int_0^{\infty} f(y) \nu(\dd y). 
  $$
\end{theorem}
\begin{remark}
  We stress that the convergence in Theorem \ref{Th2}
  can often be significantly strengthened. 
  More precisely, when $Y$ is positive recurrent,
  it is often possible to show by a classical coupling argument, that the weak convergence
  $$\QQ_{x_0}(Y_t \in \dd y) \Longrightarrow 
    \frac{\dd y}{c(y) |L'_{y,y}(\rho)|} $$
  actually holds in the total variation sense.
  Further, when there is a spectral gap, the convergence takes place exponentially fast. 
  See, for instance,  \cite{Hai-conv,MT-book,MT92,MT93a,MT93b}
  for general results in this field.
  It should be plain from the proof below that these properties 
  can then be transferred to the fragmentation semigroup. 
  We will go into more detail on this topic in the next section,
  in the special case when the growth rate $c$ is linear.
\end{remark}

\begin{proof}[Proof (of \autoref{Th2})]
The Feynman-Kac solution to the growth-fragmentation equation given 
  in Lemma \ref{Le1} can be now expressed in terms of $(Y_t)_{t\geq 0}$ as
  $$T_tf(x) =\e^{\rho t} \bar \ell(x)\QQ_x\left( f(Y_t)/\bar\ell(Y_t)\right).$$
   Recall from Corollary \ref{Co2}(ii) that  $Y$ is positive 
  recurrent whenever \eqref{e:+R} holds, and we conclude that
  $$\lim_{t\to \infty} \e^{-\rho t} T_tf(x) 
    = \bar \ell(x) \int_0^{\infty} \frac{f(y)}{\bar \ell (y)}
    \times \frac{1}{c(y) |L'_{y,y}(\rho)|}
    \ \dd y = \bar \ell(x) \ip{\nu}{f}.
  $$ 
\end{proof}

\begin{remark}\label{r:ratio-limit-theorem}
In the same  vein, it might be interesting to point at a similar application of the ratio limit theorem for point recurrent Markov processes (see, for instance, \cite[Corollary 20.8]{Kal}
  for a statement of this theorem in discrete time) which holds also  in the null recurrent case. Specifically, 
  assume \eqref{e:rt0} holds. Then, for every $f,g\in{\mathcal C}_c$ with $g\geq 0$ and $g\not \equiv 0$,
    and every $x>0$, we have
    \[
      \lim_{t\to\infty}
      \frac{\int_0^t e^{-\rho s} T_s f(x)\, \dd s}
      {\int_0^t e^{-\rho s}T_s g(x)\, \dd s}
      =
      \frac{\ip{m_0}{f/\bar\ell}}{\ip{m_0}{g/\bar\ell}}.
    \]

\end{remark}

We now conclude this section by observing that the asymptotic profile
$\nu$ is an eigenmeasure with eigenvalue $\rho$ of the 
growth-fragmentation operator ${\mathcal A}$, at least under some mild assumptions. 
In this direction, recall that ${\mathcal A}\bar f(x)= x\bar {\mathcal A}f(x)$,
where $\bar f(x)=xf(x)$ and $f\in{\mathcal D}(\bar {\mathcal A})$. 

 \begin{proposition}
 \label{P6} Assume \eqref{e:+R} holds and that $\ell$ is bounded away from $0$ 
 on $(0,\infty)$. Then 
 $\nu$   is an eigenmeasure of the dual operator ${\mathcal  A}^*$ 
 of ${\mathcal  A}$, with eigenvalue $\rho$, that is  
 $\ip{\nu}{{\mathcal  A}\bar f} = \rho \ip{\nu}{\bar f}$
 for every $f\in{\mathcal D}(\bar {\mathcal A})$. 
 \end{proposition}
\begin{proof} 
  Setting $ \bar \nu (\dd y) = y \nu(\dd y)$,  we need to check that
  $\ip{\bar \nu}{\bar {\mathcal  A} f} = \rho \ip{\bar \nu}{ f}$
  for every function $f\in{\mathcal D}(\bar {\mathcal A})$.
  Because $\nu$ is proportional to $m/\bar \ell$,  it suffices to prove 
  the identity with $m/\ell$ replacing $\bar \nu$.
  Further,  $\bar {\mathcal A}f={\mathcal G}f +\underline c f$,
  where ${\mathcal G}$ is the infinitesimal generator of $X$. 
  So we have to verify that
  $$
    \ip{m/\ell }{{\mathcal G}f +\underline c f -\rho f}=
    0\qquad \hbox{ for every }
    f\in{\mathcal D}({\mathcal G})={\mathcal D}(\bar {\mathcal A}).
  $$
  That is, using the notation $\mathcal{G}_Y$, defined in \eqref{e:gY},
  for the generator of $Y$,
  we must show
  \begin{equation}\label{e:vp}
    \ip{m }{{\mathcal G}_Y(f/\ell  )}=0
    \qquad \hbox{ for every }
    f\in{\mathcal D}({\mathcal G}).
  \end{equation}

  If we set $g = f/\ell$, then the process given earlier in \eqref{e:qmart}
  is a $\QQ_x$-local martingale.
  Moreover, it remains so when stopped at $H_Y(x)$.
  If we assume that $\ell$ is bounded away from $0$ on $(0,\infty)$, then 
  both $g$ and ${\mathcal G_Y}g$ are bounded. Recall further that the 
  occupation measure $m_0$ of the excursion of $Y$ away from $0$ is finite, 
  since thanks to Corollary \ref{Co2}, \eqref{e:+R} ensures that $Y$ is 
  positive recurrent. We deduce from the optional sampling theorem that 
  $$\QQ_{x_0}\left(\int_0^{H_Y(x_0)} {\mathcal G_Y}g(Y_s)\dd s\right) = 0,$$
  that is, by definition of $m_0$, \eqref{e:vp} holds. 
\end{proof}

For the sake of completeness, we mention the following simple result 
which ensures that $\ell$ remains bounded away from $0$ on $(0,\infty)$. 
We omit the proof, since it is a straightforward modification of
that of Lemma \ref{Le:lb}.

\begin{lemma} Assume  that
 $$\liminf_{x\to 0+}\underline c(x)>\rho\quad \hbox{and} \quad \liminf_{x\to \infty}\underline c(x)>\rho.$$
  Then $\inf_{(0,\infty)} \ell >0$. 
\end{lemma}

\section{The case of linear growth rate}
\label{s:levy}

We shall now discuss in detail the simple case when  the function $c$ is linear, namely
$$c(x)={a} x\,, \qquad x>0,$$
for some ${a} >0$. We stress that is equivalent to requesting that
the identity function is an eigenfunction of ${\mathcal A}$ with eigenvalue ${a}$,

We first consider the case in which $X$ is recurrent.
Then, $\PP_{x_0}(H(x_0)<\infty)=1$,
and we see that \eqref{e:rt0} holds with $\rho = a$. Hence  ${\mathcal	E}_t\equiv \e^{ at}$ and 
 the semigroup $T_t$ representing the solution to the growth-fragmentation
equation \eqref{e:mu} is simply given by
\[ T_t f(x) = x \e^{at} \EE_x[  f (X_t)/X_t], \qquad f\in \bar{\mathcal{C}}_c, \quad x>0. \]
Even more, $\ell(x)\equiv 1$, and 
the  martingale multiplicative functional is trivial, namely
${\mathcal M}_t\equiv 1$, and so we have $Y=X$. 
As a consequence, if $X$ is also positive recurrent and thus possesses a
(unique) stationary distribution, say $\sigma$, then we have the convergence
\begin{equation}\label{e:cvexp}
  \lim_{t\to\infty} e^{-at} T_t f(x) = x\ip{\nu}{f},\quad \hbox{with } \nu(\dd  y) = y^{-1}\sigma(\dd y)
\end{equation}
for all continuous $f$ with compact support, as we showed in \autoref{Th2}.

In this case, the main difficulty is therefore to provide
explicit criteria, in terms of $k$, to ensure that $X$ is positive
recurrent, or even exponentially ergodic. There is a wealth
of literature concerning such conditions, with the main
technique being the application of
so-called Foster--Lyapunov criteria.
A good introduction to the field may be found in \citet{Hai-conv},
and the classic monograph of \citet{MT-book} gives a thorough
grounding in the discrete-time setting.
The basic notions have been applied and extended many times;
as a sample, \cite{MT93b} discusses storage models and queues,
\cite{BCG-rate} looks at the example of kinetic Fokker-Planck
equations, and \cite{HMS-asym} studies stochastic delay
equations and the stochastic Navier--Stokes equations.

Recently, \citet{Bou} made a study of the conservative growth-fragmentation
equation, which is closely related to our equation \eqref{e:gfe}.
Among several interesting results, he studied the asymptotic behaviour of
solutions by means of Foster--Lyapunov techniques.
Some of the key assumptions in \cite{Bou} are as follows:
\begin{assumption}\label{a:bou}
  \begin{enumerate}
    \item\label{i:bou:irred}
    $K(x) >0$ and $c(x) > 0$ for all $x >0$.
    \item\label{i:bou:asym}
    There exist constants $\beta_0,\beta_\infty,\gamma_0,\gamma_\infty$ such that
    \begin{equation}\label{e:bou:beta}
      K(x) \sim \beta_0 x^{\gamma_0} \text{ as }x\to 0
      \qquad\text{and}\qquad
      K(x) \sim \beta_\infty x^{\gamma_\infty} \text{ as } x\to\infty.
    \end{equation}
%     \vspace{-2.9ex}
%     \begin{equation}\label{e:bou:tau}
%       c(x) \sim \tau_0x^{\nu_0} \text{ as } x\to 0
%       \qquad\text{and}\qquad
%       c(x) \sim \tau_\infty x^{\nu_\infty} \text{ as }x\to \infty.
%     \end{equation}
    \item\label{i:bou:M}
    If we define
    \[ M_x(s) \coloneqq \frac{1}{K(x)}\int_0^x (y/x)^{s}\, \bar{k}(x,y) \, \dd y
      \qquad \text{and}\qquad
      M(s) \coloneqq \sup_{x>0} M_x(s),
    \]
    then there exist $A>0$ such that $M(A)<1$, and $B>0$ such that
    $M(-B)<\infty$.
  \end{enumerate}
\end{assumption}
% We remark that point \ref{i:bou:irred} is necessary for
% irreducibility of the process $X$, and so it is implicit in our assumptions.
Of course,
some restrictions on the exponents in point \ref{i:bou:asym}
are imposed by our assumptions \eqref{e:c-bound} and \eqref{H8},
and these will be made explicit below.

The methods of Bouguet are natural to apply in our situation, and the arguments
carry over with minimal modifications. We therefore present in the following
result a sufficient criterion for exponential ergodicity, which is the
strongest case; weaker assumptions can be made in order to show only
ergodicity, and we refer to \cite{Bou} for more details.

% \medskip\noindent
For the result below, recall that by the Riesz representation
theorem, for every $x>0$,
there exists a family of
measures $(\mu_t^{x})_{t\ge 0}$ with the property that
$\ip{\mu_t^x}{f} = T_t f(x)$ for any continuous, compactly supported
function $f\from (0,\infty)\to\RR$. Moreover, the
measures $yx^{-1} \e^{-at} \mu_t^x(\dd y)$ are probability
measures. Finally, we recall the definition of the
\emph{total variation distance} between two probability
measures $P$ and $Q$ on $(0,\infty)$ 
as being given by
\[ \mathrm{d}_{\mathrm{TV}}(P,Q) = \frac{1}{2}\sup\{ \abs{P(B)-Q(B)} : B\subset (0,\infty), B\  {\hbox{Borel set}}\}. \]
This discussion permits us to state the following result:

% Bouguet assumes that $V$ is convex as well as smooth, but
% I don't think it is necessary. He also makes some other
% assumptions on $a$ and $b$ which I believe are no longer
% required if we directly apply [MT93b, Theorem 6.1].
\begin{proposition}\label{Pr1}%
  Suppose $c(x)={a} x$ for some ${a} >0$ and that \autoref{a:bou} is in place.
  Furthermore, assume that
  $\gamma_\infty = 0$ and $a/\beta_\infty < (1-M(A))/A$,
  and that either $\gamma_0>0$ or else $\gamma_0=0$ and
  $a/\beta_0 < (M(-B)-1)/B$.
%   $\gamma_0\ge 0$, $\gamma_\infty\le 0$,
%   $\nu_0\ge 1$ and $\nu_\infty\le 1$, and that
%   $\gamma_0>\nu_0-1$ and $\gamma_\infty>\nu_\infty-1$.
  Let $V\from(0,\infty)\to(0,\infty)$ be a smooth function such that $V(x) = x^{-B}$
  for $x\le 1$ and $V(x) = x^A$ for $x\ge 2$.
  
  Then, the Markov process $X$ has a unique stationary distribution $\sigma$.
  There exist two constants $\varepsilon>0$ and  $C<\infty$ such that,
  for every $x>0$, the semigroup $T_t$ giving the solution
  of the growth-fragmentation \eqref{e:mu} has the following asymptotic
  behaviour: 
  \[ \dd_{\rm TV}\left( \e^{-{a} t} \frac{y}{x}\mu_t^x(\dd y),\sigma(\dd y)  \right)
   \leq C(1+V(x)) \e^{-\varepsilon t}\,.
  \]
\end{proposition}
\begin{proof}
  We summarise the main points of the proof, which \citet{Bou} gives
  in greater detail. The idea is to show that
  the Markov process $X$ is exponentially ergodic, using the results
  of \cite[Theorem 6.1]{MT93b}. Thus, in the terminology of that work,
  we need to show that
  compact subsets of $(0,\infty)$ are petite for $X$, that $V$ is a norm-like
  function, and that there exist $\alpha,\delta>0$ such that
  \begin{equation}\label{e:Lyap}
    \mathcal{G}V(x) \le -\alpha V(x) + \delta.
  \end{equation}
  The petiteness of compact sets is shown in \cite[p.~6]{Bou}, and requires
  nothing more than the fact that, on compact subsets
  of $(0,\infty)$, $c$ is bounded away from zero and infinity
  and $K$ is bounded away from infinity.
  The condition that $V$ be norm-like entails that $V(x) \to\infty$
  as $x\to0$ or $x\to\infty$, which is plainly true, as well as that
  it is in the domain of the generator.
  
  The condition \eqref{e:Lyap} requires the more stringent conditions
  on the asymptotic exponents and the existence of values $A$ and $B$.
  We briefly describe the argument. For $x\ge 2$, we have
  \begin{eqnarr*}
    \mathcal{G}V(x)
    &\le&
    \bigl\{ aA - K(x)
    \bigl( 1-M_x(A) - M_x(-B)x^{-(A+B)} - R x^{-A} \bigr) \bigr\} V(x),
  \end{eqnarr*}
  where $R = \min_{x\in [1,2]} V(x)>0$;
  and for $x\le 1$, we have
  \[ \mathcal{G}V(x)
    \le \bigl\{ -aB + K(x) \bigl(M_x(-B)-1\bigr) \bigr\} V(x)
  \]

  In the case $x\ge 2$, the term within braces is equal to
  $aA-K(x)(1-M(A)+o(1))$, and as $x\to\infty$, this converges
  to a negative constant precisely when $a/\beta_\infty < (1-M(A))/A$.
  Similarly, in the case $x\le 1$, the term in braces is bounded by
  a negative constant when $x$ is close enough to zero,
  provided the conditions of the theorem hold.
  Since $V$ is bounded on compact subsets of $(0,\infty)$,
  this implies that \eqref{e:Lyap} holds, and so
  \cite[Theorem~6.1]{MT93b} completes the proof.
\end{proof}
\begin{remark}
  The reader who compares our result to \cite{Bou} will notice
  that many cases in the latter work are not accommodated by our
  assumptions. The most significant difference is that, in
  \cite{Bou}, the fragmentation rate $K$ may be unbounded.
  Giving a version of \autoref{Pr1} in this case would involve
  only a minor adaptation of the proof, but several earlier
  results of this work, such as the identification
  of the eigenmeasure $\nu$ of $\mathcal{A}^*$ in \autoref{P6},
  would become
  significantly more difficult. Since our main goal
  in this article is to point out connections with spectral
  theory, we prefer not to stray too far from the situation
  where such results may be proved.
\end{remark}

% \medskip\noindent
We shall next discuss the situation when $X$ is transient,
in which we observe different asymptotic behaviour. In this
part, we shall focus on the case where the fragmentation kernel is \emph{homogeneous}, in the sense that
$$\bar{k}(x,y) = y^{-1} \pi(\log(y/x)) \quad\hbox{for some function } \pi \in L_+^1((-\infty,0)).$$

Then the operator $\mathcal{G}$ is given by
\[
  \mathcal{G} f(x) = ax f'(x) + \int_0^x( f(y)-f(x)) \pi(\log(y/x)) y^{-1} \, \dd y ,
  \qquad f\in D(\bar{\mathcal{A}}).
\]
Our analysis will hinge on the observation that $\mathcal{G}$ can be
related to the generator of a L\'evy process, as we shall shortly
make clear.

The growth-fragmentation equation given by the corresponding
operator $\mathcal{A}$  was studied in \cite{Haas1, DouEsc, BW-gfe},
among others. Indeed,
the process $X$ corresponds to the so-called `tagged fragment'
in a random particle model, as we briefly described in
\cite[\S 6]{BW-gfe}. Homogeneous growth-fragmentation equations
are often studied
via a `cumulant function' $\kappa$, which is defined as follows.
For $\theta \in \RR$, we define $h_\theta\from(0,\infty)\to\RR$ by
$h_{\theta}(x) = x^{\theta}$, and then $h_{\theta}$
is an eigenfunction of (an extension of) ${\mathcal A}$ with eigenvalue $\kappa(\theta)$;
that is, 
$\mathcal{A}h_\theta = \kappa(\theta) h_\theta$. 
The function $\kappa$ can be given explicitly as
\[ \kappa(\theta) = a\theta + \int_0^1 (y^{\theta-1}-1) \pi(\log y) y^{-1} \,\dd y,
  \qquad \theta \in \RR,
\]
and it is smooth and strictly convex.
Our basic assumption, for the remainder of this section, is that
there exists some $\theta_0 \ne 1$, lying in the interior
of the domain of $\kappa$, with the property that
$\kappa'(\theta_0) = 0$. Observe that in particular,
$\kappa(\theta_0)=\min_{\theta \in \RR} \kappa(\theta)$.

We now look more closely at  $X$, and introduce the 
following auxiliary process, which is a L\'evy process;
for further background on this class of processes,
we refer to \cite{Ber-Levy,Kyp,Sato}.
Consider a Lévy process $\xi$
composed of a compound Poisson process with negative jumps plus a drift
$a>0$, and such that $\xi$
has an absolutely continuous Lévy measure with density $\pi$.
Let $\psi$ represent the Laplace exponent of this Lévy process,
which means that $\EE[\e^{\theta \xi_t}\mid \xi_0=0] = \e^{t\psi(\theta)}$.
This function is smooth and strictly convex, with Lévy--Khintchine
representation as follows:
\[ \psi(\theta) = a\theta + \int_{-\infty}^0 (\e^{\theta u}-1) \pi(u)\,\dd u
  = a\theta + \int_{0}^1 (u^{\theta}-1) \pi(\log u) u^{-1} \, \dd u,
  \qquad \theta\in\RR. \]
It is 
related to $\kappa$ via the equation $\psi(\theta) = \kappa(\theta+1)-\kappa(1)$,
from which we see that $\theta_0$ satisfies
$\psi'(\theta_0-1) = 0$.
The existence of
$\theta_0$ implies that $\psi'(0) = \EE[\xi_1 \mid \xi_0=0] \ne 0$,
which means that either $\lim_{t\to\infty} \xi_t = \infty$
or $\lim_{t\to\infty} \xi_t = -\infty$. In particular,
$\xi$ is a transient process.

By comparing $\mathcal{G}$ with the generator
of a Lévy process \cite[Theorem 31.5]{Sato},
$X$ may be identified as
\[ X_t = \e^{\xi_t}, \qquad t\ge 0, \]
and so $X$ is also transient.

A natural component of our analysis in this situation
is the inverse function $\Phi$ of $\psi$, defined by
$\Phi(q) = \sup\{ \theta\in\RR: \psi(\theta) = q\}$.
It appears in the following expression,
in which $\tau(0) = \inf\{t > 0 : \xi_t = 0\}$:
\[ \EE[ \e^{-q \tau(0)} ;\tau(0)<\infty \mid \xi_0 = 0] = 1-\frac{1}{\Phi'(q)}. \]
This formula can be found, for instance, in Lemma 2(i) of \citet{PardoPR}.

From this, we can calculate the spectral radius of the growth-fragmentation
equation associated with $\mathcal{G}$. Since the return time of
$\xi$ to its starting point is equal to that of $X$, we calculate,
using the inverse function theorem,
\[ L_{x_0,x_0}(q) = 1-\frac{1}{\Phi'(q-a)} = 1-\psi'(\Phi(q-a)). \]
This implies that
$\rho = \kappa(\theta_0) = a + \psi(\theta_0-1) < a$, so that contrary
to the situation where $X$ is recurrent, here
the spectral radius is strictly less than the drift coefficient $a$.

Moreover,
\[ -L_{x_0,x_0}'(q) = \frac{\psi''(\Phi(q-a))}{\psi'(\Phi(q-a))} , \]
and as $q \downarrow \rho$, we obtain, by the
strict convexity of $\psi$, that $-L_{x_0,x_0}'(\rho) = \infty$.
Thus, we are in a situation where  the process $Y$ is \emph{null} recurrent.

We now study the function $\ell$ in more detail.
In order to compute it explicitly, we recall (from
\cite[\S 3.3]{Kyp}, for instance) that the process
$(\e^{(\theta_0-1)\xi_t - t\psi(\theta_0-1)})_{t\ge 0}$ is a non-negative
martingale.
Since $\lim_{t\to\infty} \xi_t/t = \psi'(0) \ne 0$
almost surely (see \cite[Exercise 7.2]{Kyp}),
the martingale converges almost surely to $0$ as $t\to\infty$.
We obtain the following explicit formula for $\ell$,
applying in the third equality
the optional sampling theorem \cite[Theorem~II.77.5]{RW1} at $H(\log x_0)$.
\begin{eqnarr*}
  \ell(x) = L_{x,x_0}(\rho)
  &=&
  \EE[ \e^{-(\rho-a)H(\log x_0)} ; H(\log x_0) < \infty \mid \xi_0 = \log x] \\
  &=&
  \EE[ \e^{(\theta_0-1)\log(x_0) -\psi(\theta_0-1)H(\log x_0)} ; H(\log x_0) <\infty
  \mid \xi_0 = \log x]
  \e^{-(\theta_0-1)\log(x_0)} \\
  &=& \e^{(\theta_0-1)(\log x - \log x_0)}
  \\
  &=& (x/x_0)^{\theta_0-1}.
\end{eqnarr*}

Furthermore, we can calculate directly from \eqref{e:gY} that the generator
of $Y$ is given by
\[ \mathcal{G}_Y g(x)
  = ax g'(x) + \int_{0}^x \bigl(g(y)-g(x) \bigr) (y/x)^{\theta_0-1} \pi(\log(y/x))
  \, \frac{\dd y}{y}.
\]
In other words, we have the representation $Y_t = \exp(\eta_t)$, where
$\eta$ is a Lévy process whose Laplace 
exponent is given by $\theta \mapsto \psi(\theta+\theta_0-1)-\psi(\theta_0-1)$.
This Lévy process has the property that $\EE[\eta_1 \mid \eta_0=0] = 0$, which
implies that $\eta$ is recurrent (see, for instance, \cite[Remark 37.9]{Sato}.)

Finally, we wish to study the asymptotic behaviour of the semigroup $T_t$, or equivalently,
the measures $\mu_t^x$ introduced earlier. The semigroup can be identified explicitly
in terms of our L\'evy process $\eta$ as:
\[ 
  T_tf(x)
  = 
  \e^{\rho t} \bar \ell(x) \QQ_x\bigl[ f(Y_t)/\bar \ell(Y_t)\bigr ]
  = \e^{\kappa(\theta_0) t} x^{\theta_0}
  \EE\bigl[ f(\e^{\eta_t})\e^{-\theta_0\eta_t}\mid \eta_0=\ln x\bigr].
\]
The asymptotics of this semigroup could be studied using \autoref{r:ratio-limit-theorem}.
However, more precise information can be obtained by applying
instead a local central limit theorem for $\eta$ (see \cite[Theorem 8.7.1]{Bor-pt}.)
In this way, one recovers the formula
\[
  T_t f(x)
  \sim \frac{x^{\theta_0} e^{t \kappa(\theta_0)}}{\sqrt{2 \pi t \kappa^{\prime\prime}(\theta_0))}}
  \int_0^\infty f(y) y^{-(\theta_0+1)} \, \dd y,
  \text{ as } t\to\infty,
\]
for $f$ continuous and compactly supported, which was stated as
\cite[Corollary 3.4]{BW-gfe}, under different
assumptions.

\bibliography{gfe}

\end{document}